\documentclass[reqno]{amsart}

\usepackage{amsmath,amssymb,amsthm,amsfonts,
	enumerate,hyperref,cleveref}

\hypersetup{
	colorlinks=true,
	linkcolor=blue,
	filecolor=magenta,      
	urlcolor=cyan,
	pdftitle={Overleaf Example},
	pdfpagemode=FullScreen,
}
\usepackage{dsfont}
\usepackage[all]{xy}
\usepackage[T1]{fontenc}
\usepackage{tikz}
\usepackage{pgfplots}
\pgfplotsset{compat=1.15}
\usepackage{mathrsfs}
\usetikzlibrary{arrows}
\newtheorem{lemma}{Lemma}[section]
\newtheorem{theorem}[lemma]{Theorem}
\newtheorem{proposition}[lemma]{Proposition}

\crefrangeformat{equation}{#3(#1)#4--#5(#2)#6}

\crefname{thrm}{Theorem}{Theorems}
\crefname{theorem}{Theorem}{Theorems}
\crefname{lem}{Lemma}{Lemmas}
\crefname{cor}{Corollary}{Corollaries}
\crefname{prop}{Proposition}{Propositions}
\crefname{defn}{Definition}{Definitions}
\crefname{exm}{Example}{Examples}
\crefname{rem}{Remark}{Remarks}
\crefname{conj}{Conjecture}{Conjectures}
\crefname{quest}{Question}{Questions}
\crefname{section}{Section}{Sections}
\crefname{equation}{\unskip}{\unskip}
\crefname{enumi}{\unskip}{\unskip}
\crefname{subsection}{Subsection}{Subsections}

\begin{document}
\title{A Le Page--Kaplansky theorem characterizing commutative JB$^*$-triples}	

\author[L. Li]{Lei Li}
\address[L. Li]{School of Mathematical Sciences and LPMC, Nankai University, 300071 Tianjin, China.}
\email{leilee@nankai.edu.cn}

\author[S. Liu]{Siyu Liu}
\address[S. Liu]{School of Mathematical Sciences and LPMC, Nankai University, 300071 Tianjin, China.}
\email{760659676@qq.com}

\author[A.M. Peralta]{Antonio M. Peralta}
\address[A.M. Peralta]{Instituto de Matem{\'a}ticas de la Universidad de Granada (IMAG), Departamento de An{\'a}lisis Matem{\'a}tico, Facultad de
	Ciencias, Universidad de Granada, 18071 Granada, Spain.}
\email{aperalta@ugr.es}

\subjclass[2010]{Primary 46J99; 47A30; 17C65 Secondary 46H70; 17A15}
\keywords{Le Page-type theorem, JB$^*$-triple, C$^*$-algebra, commutativity, JB$^*$-algebra, associativity.} 
	
\begin{abstract} We prove that a Le Page-type inequality is also valid for metrically characterizing those JB$^*$-triples that are commutative. More precisely, we establish that the following statements are equivalent for any JB$^*$-triple $E$:

$(a)$ $E$ is commutative. 

$(b)$ There exists $\gamma>0$ satisfying $$\big\|\{a,b,\{x,y,z\}\}\big\|\leq \gamma \ \! \big\|\{x,y,\{a,b,z\}\}\big\|, \hbox{ for all } a,b,x,y,z\in E.$$ 
\end{abstract}
	
	\maketitle
	
	
\section{Introduction} 

Le Page theorem is a starring result in the theory of associative normed algebras, it affirms that a normed unital associative complex algebra $A$ is commutative if, and only if, there exists $\gamma>0$ satisfying $\| a b \| \leq \gamma \ \! \| b a \|,$ for all $a,b\in A$ (see \cite{LePage}, \cite[Theorem 6.3.5]{CabRodVol2}, \cite[\S 2.1, Theorem 1]{Aupetit79} and the recent paper \cite{Pop25}). It is known that the conclusion fails when $A$ is not unital. Consider, for example, the matrices ${e_1} = \begin{pmatrix} 
	0 & 0 & 0 & 0 \\ 
	1 & 0 & 0 & 0 \\ 
	0 & 0 & 0 & 0 \\ 
	0 & 0 & 1 & 0 
\end{pmatrix},$  ${e_2} = \begin{pmatrix} 
	0 & 0 & 0 & 0 \\ 
	0 & 0 & 0 & 0 \\ 
	1 & 0 & 0 & 0 \\ 
	0 & -1 & 0 & 0 
\end{pmatrix}$, and $e_1 e_2 = \begin{pmatrix} 
0 & 0 & 0 & 0 \\ 
0 & 0 & 0 & 0 \\ 
0 & 0 & 0 & 0 \\ 
1 & 0 & 0 & 0 
\end{pmatrix} = - e_2 e_1$ in $M_{4} (\mathbb{C})$. It is easy to see that $A = \hbox{span}\{e_1,e_2, e_1 e_2\}$ is a three dimensional closed subalgebra of $M_4 (\mathbb{C})$ which is anti-commutative (i.e. $ a b = -b a$ for all $a,b\in A$) but not commutative. 
Therefore, $\| a b\| = \|-b a\| = \| b a \|$, for all $a,b\in A$.\smallskip

Another classical result, historically attribute to Kaplansky, assures that if $A$ is a (not necessarily unital) C$^*$-algebra, then $A$ is not commutative if, and only if, there exists a non-zero $x\in A$ with $x^2 =0$ (see \cite[2.12.21]{diximer77C*-algebras}). Note that in such a case, $x (x^* x )\neq 0$, while $(x^* x) x =0$.  So, the previous counterexample does not exist in the setting of C$^*$-algebras, and thus, for a C$^*$-algebra $A$ the existence of $\gamma>0$ satisfying $\| a b \| \leq \gamma \ \! \| b a \|,$ for all $a,b\in A$, is equivalent to the commutativity of $A$.\smallskip

A complex \emph{Jordan algebra} is a (non-necessarily associative) algebra $\mathfrak{A}$ whose product (denoted by $\circ$) is commutative, and instead of associativity, the product satisfies the \emph{Jordan identity}:  $(a \circ
b)\circ a^2 = a\circ (b \circ a^2)$ ($a,b\in \mathfrak{A}$). If $\mathfrak{A}$ admits a norm $\|. \|$ satisfying $\| a\circ b\| \leq \|a\| \ \|b\|$, $a,b\in \mathfrak{A}$, we say that $\mathfrak{A}$ is a complex normed Jordan algebra. We say that $\mathfrak{A}$ is unital if it admits an element $\mathbf{1}$ satisfying $\mathbf{1}\circ a = a$, for all $a\in \mathfrak{A}$. A Jordan algebra $\mathfrak{A}$ is called associative if its Jordan product is associative (i.e.,  $(a \circ
b)\circ c =  a \circ (b\circ c)$), in such a case $\mathfrak{A}$ is associative and commutative. Having in mind that every Jordan algebra is power-associative (i.e. all subalgebras generated by a single element are associative \cite[Proposition 2.4.19]{CabRodBookV1}), it follows from \cite[Proposition 6.3.40]{CabRodVol2} that a unital complex normed Jordan algebra $\mathfrak{A}$ is associative if, and only if, there exists $\gamma > 0$ satisfying $\| (a \circ
b)\circ c\| \leq \gamma\  \| a \circ (b\circ c)\|,$ for all $a, b, c \in \mathfrak{A}$. The equivalence also holds for not necessarily unital JB$^*$-algebras (Jordan alter ego of C$^*$-algebras). We recall that a \emph{JB$^*$-algebra} is a complex Jordan-Banach algebra $\mathfrak{A}$ equipped with an algebra involution ``$^*$'' satisfying  $\|U_a ({a^*})\| = \|a\|^3$, for all $a\in \mathfrak{A},$ where for $a,b\in \mathfrak{A}$ we set $U_a (b) =  2 (a\circ b) \circ a - a^2 \circ b$. Every C$^*$-algebra $A$ is a JB$^*$-algebra when we replace the original product with the natural product defined by $a \circ b = \frac12 (a b^* + b^* a)$ ($a,b\in A$). A \emph{JC$^*$-algebra} is a norm-closed self-adjoint Jordan subalgebra of a C$^*$-algebra. There are examples of JB$^*$-algebras which are not JC$^*$-algebras (cf. \cite[Example 3.1.56]{CabRodBookV1}).\smallskip

The following Jordan version of Kaplansky's theorem discussed above is due to Iochum, Loupias and Rodríguez-Palacios (see \cite[Theorem 1]{IoLouRod1989commutativity}).

\begin{theorem}\label{thm: associative JB*-algebra nilpotent}
	Let $\mathfrak{A}$ be a (not necessarily unital) JB$^*$-algebra. Then, the following statements are equivalent:
	\begin{enumerate}[$(a)$] \item $\mathfrak{A}$ is associative.
	\item $\mathfrak{A}$ does not contain $2$-nilpotent elements.
	\item There exists $\gamma >0$  satisfying $\| (a \circ
	b)\circ c\| \leq \gamma\ \! \| a \circ (b\circ c)\|,$ for all $a, b, c \in \mathfrak{A}$.
	\end{enumerate}
\end{theorem}

As commented above, the equivalence $(a)\Leftrightarrow (b)$ is explicitly proved in \cite[Theorem 1]{IoLouRod1989commutativity} (see also \cite[Theorem 6.1.17]{CabRodVol2}), and $(a)\Rightarrow (c)$ is clear. Suppose, finally, that $(c)$ holds. If $\mathfrak{A}$ is not associative, by the equivalence $(a)\Leftrightarrow (b)$, there exists $0\neq a$ with $a^2=0$. By the axioms of JB$^*$-algebras we have $0\neq \|a\|^3 = \|U_{a} (a^*)\|$. Note that, by assumptions, we have $$ U_a (a^*) = 2 (a\circ a^*)\circ a - a^2 \circ a^* =  2 (a\circ a^*)\circ a.$$ Therefore $0\neq \| (a\circ a^*)\circ a \|$, while $\| (a\circ a)\circ a^*\| =0$, which contradicts the statement in $(c)$.\smallskip

As C$^*$-algebras define a strict subclass inside JB$^*$-algebras, the latter define a another strict subset within the wider setting of JB$^*$-triples. Formally speaking, a JB$^*$-triple is a complex Banach space $E$ equipped with a continuous triple product $\{\cdot,\cdot,\cdot\}:E\times E\times E\rightarrow E$, which is linear and symmetric in the outer variables and conjugate linear in the middle one, and satisfies the following conditions:
\begin{enumerate}[$(i)$]
	\item The (bounded linear) operator $L(a,b)$ on $E$ given by $L(a,b)(x)=\{a,b,x\}$ satisfies 
	$$L(a,b)L(x,y)-L(x,y)L(a,b)=L(L(a,b)x,y)-L(x,L(b,a)y),$$ for all $a,b,x,y\in E$. \hfill (Jordan identity)
	\item For each $a\in E$, $L(a,a)$ is a hermitian operator $\big($i.e., $\left\| e^{i t L(a,a)} \right\|=1$ for all real $t \big)$ with non-negative spectrum.
	\item $\|\{a,a,a\}\|=\|a\|^3$ for every $a\in E$. \hfill (extended Gelfand--Naimark axiom)
\end{enumerate} C$^*$-algebras and JB$^*$-algebras can be naturally viewed as JB$^*$-triples for the triple products defined by $\{x,y,z\} :=\frac12 (x y^* z+ z y^*x)$ and $\{x,y,z\} := (x\circ y^*) \circ z + (z\circ y^*)\circ x - (x\circ z)\circ y^*$, respectively (cf. \cite{Kaup1983riemann} and \cite[\S 2.2.27 and \S 4.1.3]{CabRodBookV1}).\smallskip

A JB$^*$-triple $E$ is called \emph{commutative} or \emph{abelian} if $$[L(a,b),L(x,y)]=L(a,b)L(x,y)-L(x,y)L(a,b)=0,$$ for all $a,b,x,y\in E,$ that is, $L(a,b)$ and $L(x,y)$ commute in the associative Banach algebra $B(E)$ of all bounded linear operator on $E$ (cf. \cite{Kaup77MathAnn,Kaup1983riemann, hornIdeals, FriRuCommutative}, \cite[\S 4]{DiTi88}, and more recently \cite[\S 2]{CaPe24}). One of the founding results in JB$^*$-triple theory, established by Kaup in \cite[\S 1]{Kaup1983riemann} (see also \cite[\S 4.2.1]{CabRodBookV1}), shows how every commutative JB$^*$-triple $E$ identifies, via an isometric triple isomorphism (i.e., a linear bijection preserving triple products), with the closed subtriple of the function space $C_0(L)$ of all $\mathbb{T}$-homogeneous (or $\mathbb{T}$-equivariant) continuous functions on a principal $\mathbb{T}$-bundle $L$ (i.e., a subset $L$ of a locally convex Hausdorff complex linear space such that $\mathbb{T} L = L$, $0 \notin L$ and $L \cup \{0\}$ is compact). Concretely, 
\[E\cong C_0^\mathbb{T}(L):=\big\{a\in C_0(L):a(\lambda t)=\lambda a(t)\text{ for every } (\lambda,t)\in\mathbb{T}\times L\big\},\] where the latter is equipped with the supremum norm and the triple product given by $\{a,b,c\} = a \overline{b} c$. Every commutative C$^*$-algebra is a commutative JB$^*$-triple (cf. \cite[Proposition 10]{Ol75} or \cite[Lemma 3.1]{FriedmanRusso82TAMS}), but the class of commutative JB$^*$-triples is strictly wider (see \cite[Corollary 1.13]{Kaup1983riemann}).\smallskip

Clearly, every commutative JB$^*$-triple $E$ satisfies $$\big\|\{a,b,\{x,y,z\}\}\big\| = \big\|\{x,y,\{a,b,z\}\}\big\|, \hbox{ for all } a,b,x,y,z\in E. $$ What is entirely new is the question of whether a Le Page-Kaplansky type theorem holds in the setting of JB$^*$-triples. In this work, we prove that when the norm of a JB$^*$-triple satisfies a Le Page-type inequality, the JB$^*$-triple is commutative. The main result is the following theorem.

\begin{theorem}\label{t characterization of commutative JB*-triples} Let $E$ be a JB$^*$-triple. Then, the following are equivalent:
	\begin{enumerate}[$(a)$]
		\item $E$ is commutative.
		\item There exists $\gamma>0$ satisfying \begin{equation}\label{eq Le Page triples}
			\big\|\{a,b,\{x,y,z\}\}\big\|\leq \gamma \ \! \big\|\{x,y,\{a,b,z\}\}\big\|, \hbox{ for all } a,b,x,y,z\in E.
		\end{equation} 
	\end{enumerate}
\end{theorem} 

That is, a Le Page inequality involving five elements characterises commutative JB$^*$-triples. The theorem will be proved in Section~\ref{sec: commutative JB*-triples}. Before discussing the other contents of this work, we devote a few thoughts to the innovative and novel nature of this viewpoint, which makes the result appear unrelated to and inaccessible from the theorems of Le Page, Kaplansky, and Iochum–Loupias–Rodríguez‑Palacios, as well as from the usual arguments used in their proofs. 
In the setting of JB$^*$-triples, besides the absence of a binary product, exponentials, commutators, and associators, the natural strategy of embedding the problem into the associative Banach algebra $B(E)$ to apply the classical Le Page theorem is likewise not feasible. More concretely, the inequality in \eqref{eq Le Page triples} implies that $\big\| L(a,b) L(x,y)\big\| \leq \gamma \ \! \big\| L(x,y) L(a,b) \big\|$ for all $a,b,x,y\in E$. However, the Jordan identity only implies that the linear span of the set $L(E,E) := \{L(a,b) : a,b\in E\}$ is a complex Lie subalgebra of the associative Banach algebra $B(E)$ (with respect to the commutator product), but not an associative subalgebra. \smallskip

Returning to the structure of the paper, in Section~\ref{sec: inner ideals} we begin by showing that Theorem~\ref{t characterization of commutative JB*-triples} admits a simpler proof under the stronger hypothesis that $E$ is a JBW$^*$-triple (see Theorem~\ref{t characterization of commutative JBW*-triples}). In the general setting we prove that in a JB$^*$-triple $E$ satisfying the Le Page-type  inequality in \eqref{eq Le Page triples}, every inner ideal generated by a single element in $E$ must be an associative JB$^*$-algebra or equivalently, a commutative C$^*$-algebra (see Proposition~\ref{p inner ideal single generated is commutative}). Finally, we combine geometric-algebraic tools derived from the Sait\^o-Tomita-Lusin theorem for JB$^*$-triples \cite{BuFerMarPe2006} and a recent characterization of commutativity for JB$^*$-triples in terms of inner ideals established in \cite[\S 2]{CaPe24} to establish our main result.

\section{Inner ideals and commutative JBW$^*$-triples}\label{sec: inner ideals}

As commented above, the proof of the general statement in Theorem~\ref{t characterization of commutative JB*-triples} will require a sophisticated argument involving deep algebraic-geometric tools. Nevertheless, the arguments are simpler under the stronger hypothesis that $E$ is a JBW$^*$-triple (a JB$^*$-triple which is also a dual Banach space). \smallskip

We recall first some essential background in JB$^*$-triple theory. As commented in the introduction, the class of JB$^*$-triples is strictly wider than its subclasses determined by all C$^*$-algebras and all JB$^*$-algebras, where the last two classes are actually strictly ordered by inclusion. There are however many natural ways to embed JB$^*$-algebras inside a fixed JB$^*$-triple $E$. The first one is given by the Peirce 2-subspace associated with a tripotent $e$ (i.e. an element $e$ satisfying $\{e,e,e\}=e$). The eigenvalues of the operator $L(e,e)$ are contained in the set $\{0,1,2\}$, and the corresponding eigenspaces $E_j (e)$ ($j= 0,1,2$) induce a decomposition (known as the Peirce decomposition of $E$ with respect to $e$) in the form $$ E = E_{0} (e) \oplus E_{1} (e) \oplus
E_{2} (e) 
.$$ A tripotent $e\in E$ is called complete if $E_0(e)= \{0\}$. The so-called Peirce subspaces $E_j (e)$ ($j= 0,1,2$) are all JB$^*$-subtriples of $E$ and obey certain rules for the triple products among their elements, namely, 
\begin{equation*}
	\textit{(Peirce arithmetic) \ \ }\left\{
	\begin{aligned}
		& \{ {E_{i}(e)},{E_{j}(e)},{E_{k}(e)} \} \subseteq E_{i-j+k}(e),\ \hbox{
			if  } i-j+k\in\{0,1,2\}, \\ 
		& \{ {E_{i}(e)},{E_{j}(e)}, {E_{k}(e)} \} =\{0\}, \hbox{ for } i-j+k\neq 0,1,2, \\
		& \{ {E_0 (e)},{E_2 (e)},{E}\} = \{ {E_2 (e)},{E_0 (e)},{E }\} =\{0\}.
	\end{aligned}
	\right.
\end{equation*}
 For the purposes of this paper we remark that by defining a Jordan product and involution on $E_2 (e)$ by $x\circ_{e} y := \{x,e,y\}$ and $x^{{*}_{e}} := \{e,x,e\}$, respectively, it turns out that $(E_2(e), \circ_e, *_{e})$ is a unital JB$^*$-algebra with unit $e$ (cf. \cite[\S 4.2.2]{CabRodBookV1}).\smallskip
	
A general JB$^*$-triple may contain no non-zero tripotents. However, since the extreme points of the closed unit ball of each JB$^*$-triple $E$ coincide with the complete tripotents in $E$  (see \cite[Theorem 4.2.34]{CabRodBookV1}), every JBW$^*$-triple contains an abundant collection of complete tripotents.\smallskip
	
Concerning commutativity, Theorem 3 in \cite{CaPe24} proves that a JBW$^*$-triple $W$ is commutative if, and only if, for each complete tripotent $e\in W$ we have $L(e,e) (W) \subseteq W_2(e),$ equivalently, $W_1 (e)=\{0\}$ (cf. \cite[Theorem 3]{CaPe24}). We shall employ this characterization to establish the next version of Theorem~\ref{t characterization of commutative JB*-triples} for JBW$^*$-triples.

\begin{theorem}\label{t characterization of commutative JBW*-triples} Let $W$ be a JBW$^*$-triple. Then, the following are equivalent:
	\begin{enumerate}[$(a)$]
		\item $W$ is commutative.
		\item There exists $\gamma>0$ satisfying \begin{equation}\label{eq Le Page JBWstar-triples}
			\|\{a,b,\{x,y,z\}\}\|\leq \gamma \ \|\{x,y,\{a,b,z\}\}\|, \hbox{ for all } a,b,x,y,z\in W.
		\end{equation} 
	\end{enumerate}
\end{theorem}

\begin{proof} Suppose $e$ is a (complete) tripotent in $W$. If there exists a non-zero $x$ in $W_1 (e)$, it can be easily checked via Peirce arithmetic that  $$\{x,e,\{e,x,e\}\} = \{e,x, 0 \} = 0,$$ and  $$\{e,x,\{x,e,e\}\} = \left\{e,x,\frac12 x\right\} = \frac12 \{e,x,x\}.$$ Now, an application of \cite[Proposition 2.4]{BuFerMarPe2006} assures that $$0<\|x\|^2 \leq 4 \left\|\{x,x,e\} \right\|.$$ However, by hypothesis \eqref{eq Le Page JBWstar-triples}, we have $$0< \frac12 \left\| \{e,x,x\} \right\|= \left\| \{e,x,\{x,e,e\}\} \right\| \leq \gamma^2 \left\| \{x,e,\{e,x,e\}\}  \right\|=0, $$ which is impossible. Therefore, $W_1 (e) = \{0\}$ for all tripotent $e\in W$. Theorem 3 in \cite{CaPe24} assures that $W$ is commutative.
\end{proof}

Let us make some comments. First, a general JB$^*$-triple may contains no non-trivial tripotents, and the characterization of commutativity in \cite[Theorem 3]{CaPe24} makes no sense for general JB$^*$-triples. Second, Theorem~\ref{t characterization of commutative JB*-triples} does not follow from Theorem~\ref{t characterization of commutative JBW*-triples}. Although the bidual, $E^{**}$, of a JB$^*$-triple $E$ is a JBW$^*$-triple \cite{Di86}, the existence of positive $\gamma$ for which the inequality in \eqref{eq Le Page JBWstar-triples} holds for all $a,b,x,y,z\in E$ does not necessarily imply the validity of the inequality for all $a,b,x,y,z\in E^{**}$ (even applying the known facts that the triple product of $E^{**}$ is separately weak$^*$-continuous \cite{BartonTimoney1986weak} and $E$ is weak$^*$-dense in $E^{**}$).\smallskip

We shall also make use of the connections between surjective linear isometries and triple isomorphisms. A celebrated result by Kadison shows that a unital linear bijection between two C$^*$-algebras is an isometry if, and only if, it is a Jordan $^*$-isomorphisms \cite[Theorems 5 and 7]{Kad1951}. A surjective linear isometry between C$^*$-algebras which is not necessarily unital need not preserve associative or Jordan products. Nevertheless, surjective linear isometries can be characterized by an algebraic property, even in the wider setting of JB$^*$-triples, since a surjective linear mapping between JB$^*$-triples is an isometry if, and only if, it preserves triple products (cf. \cite[Proposition 5.5]{Kaup1983riemann}).\smallskip

The arguments in the proof of Theorem~\ref{t characterization of commutative JBW*-triples} actually reveal some connections between the Le Page-type inequality in \eqref{eq Le Page JBWstar-triples} with elements in the ``centroid'' of a JB$^*$-triple. Following \cite{DiTi88}, we define the \emph{centroid}, $Z(E)$, of a JB$^*$-triple $E$ as the set of all continuous linear operators $T \in B(E)$ satisfying:
$$ T(\{x,y,z\}) = \{T(x), y, z\}, \qquad \text{for all } x,y,z \in E.$$ In case that $A$ is a C$^*$-algebra (respectively, a JB$^*$-algebra) regarded as a JB$^*$-triple, the centroid of $A$ is precisely the collection of all left or right (respectively, Jordan) multiplication operator by a fixed element in the centre of $A$ (cf. \cite[Proposition 3.5]{DiTi88}). For example, a projection (i.e. a symmetric idempotent) $p$ in a JB$^*$-algebra $A$ is central (i.e., it operator commutes with every element in $A$) if, and only if, $M_p : A\to A$, $M_p (x) = p\circ x$ is an element in the centroid of $A$ when the latter is regarded as a JB$^*$-triple. Note that every projection $p\in A$ is a tripotent, and it is central if, and only if, $A_1 (p) =\{0\}$ (cf. \cite[2.5.7]{hanche-olsen84jordan}). This is the motivation to obtain the following extension of \cite[Lemma 5]{Pop25} to the setting of JB$^*$triples.

\begin{lemma}\label{l centroid tripotents} Let $e$ be a tripotent in a JB$^*$-triple $E$. Then the following statements are equivalent:\begin{enumerate}[$(a)$]\item $E_1 (e) =\{0\}$ {\rm(}equivalently, $L(e,e)= P_2(e)${\rm)}.
\item $L(e,e)$ is an element in the centroid of $E$.
\item There exists a positive $\gamma$ satisfying $$ \big\| L(e,x) L(x,e) (e) \big\| \leq \gamma \ \! \big\| L(x,e)  L(e,x) (e)  \big\|, \hbox{for all } x\in E.$$
\end{enumerate} 
\end{lemma}

\begin{proof} The equivalence $(a)\Leftrightarrow (b)$, which is almost explicit in \cite[\S 4]{DiTi88}, follows from Peirce arithmetic. If $(a)$ holds, given $x,y,z\in E$ we write $x = x_0 + x_2$, $y = y_0 + y_2$, and $z = z_0 + z_2$, with $x_j= P_j(e)(x),$ $y_j= P_j(e)(y),$ and $z_j= P_j(e)(z)$ ($j=0,2$). By Peirce arithmetic, $$L(e,e)\{x,y,z\} =  L(e,e)\{x_2,y_2,z_2\} = \{x_2,y_2,z_2\} =  \{x,y,L(e,e)(z)\},$$ which proves that $L(e,e)$ lies in the centroid of $E$.\smallskip
	
We still assuming $(a)$ and the previous notation. It is easy to check that $L(e,x) L(x,e) (e) = L(e,x) (x_2) = \{e,x_2,x_2\}$. On the other hand, since $e$ and $x_2$ lie in $E_2(e),$ and the latter is a JB$^*$-algebra, by \cite[Proposition 5.5]{Kaup1983riemann}, the original triple product on $E_2(e)$ coincides with the triple product as JB$^*$-algebra. Consequently, 
$$L(x,e) L(e,x) (e) = L(x_2,e) L(e,x_2) (e) = L(x_2,e) (x_2^{*_e}) =  x_2 \circ_e x_2^{*_e} = \{e,x_2,x_2\},$$ and thus  $L(x,e) L(e,x) (e) = L(e,x) L(x,e) (e)$. We have therefore shown that $(a)\Rightarrow (b)$ \& $(c)$.\smallskip

Suppose now that $(b)$ holds. If there exists $0\neq x \in E_1(e)$, by \cite[Theorem 2.3]{Pe2015}, the element $L(e,x)(x) = \{x,x,e\}$ is non-zero. A new application of Peirce arithmetic proves that the identities $\{x,x,e\} =L(e,e) \{x,x,e\} = L(e,e) L(e,x) (x) = L(e,x)(x) \neq 0$, and $ L(e,x) L(e,e) (x) = \frac12 L(e,x) (x)$ hold. Therefore $L(e,e)$ is not an element in the centroid of $E$, which contradicts the assumption. This shows that $(b)\Rightarrow (a)$.\smallskip

$(c)\Rightarrow (a)$ As before, if there exists $0\neq x \in E_1(e)$, the Peirce arithmetic leads to $0=L(x,e) L(e,x) (e)$, while $L(e,x) L(x,e) (e) = L(e,x) (x) \neq 0$ (see \cite[Theorem 2.3]{Pe2015}), which contradicts $(c)$.
\end{proof}

The remaining part of this section is devoted to studying the first algebraic consequences that can be derived for a  JB$^*$-triple $E$ satisfying the Le Page inequality in \eqref{eq Le Page triples}, that is, the existence of $\gamma>0$ satisfying $\|\{a,b,\{x,y,z\}\}\|\leq \gamma \ \! \|\{x,y,\{a,b,z\}\}\|,$ for all $a,b,x,y,z\in E.$ We shall first show how this property implies that every inner ideal generated by a single element in $E$ must be an associative JB$^*$-algebra or equivalently, a commutative C$^*$-algebra.\smallskip 

A closed subspace $I$ of a JB$^*$-triple $E$ is an \textit{inner ideal} of $E$ if $\{I,E,I \} \subseteq I.$ 
 The Peirce subspaces $E_2(e)$ and $E_0(e)$ associated with a tripotent $e\in E$ are inner ideals of $E$.\smallskip

Let us revisit some properties and notation required in our arguments. We have already surveyed the Gelfand theory for commutative JB$^*$-triples in the introduction. The representation is even simpler if the commutative JB$^*$-triple is in fact generated by a single element. More concretely, let $a$ be an element in a JB$^*$-triple $F$, and let the symbol $F_a$ stand for the JB$^*$-subtriple of $F$ generated by $a$ (i.e. the norm closure of the linear span of all odd powers of $a$ defined, recursively, by $a^{[3]}=\{a,a,a\},$ and $a^{[2n+1]} = \{a,a,a^{[2n-1]}\}$ for all $n\geq 2$). It is known that $F_a$ is identified, via an isometric triple isomorphism $\Psi_a$, with a commutative C$^*$-algebra of the form $C_0(\Omega_a),$ for a (unique) locally compact Hausdorff space $\Omega_a\subseteq [0,\|a\|]$ depending on $a$, where $\Omega\cup \{0\}$ is compact and $\Psi_a(a)$ is the natural embedding of $\Omega_a$ into $\mathbb{C}$ (cf. \cite[Corollary 1.15]{Kaup1983riemann}, \cite[\S 3]{Ka96} and \cite[Theorem 4.2.9]{CabRodBookV1}). The set $\Omega_a$ is known as the triple spectrum of $a$. If $f$ is any continuous function in $C_0(\Omega_a),$ we write $f_t (a)$ for the element $\Psi_{a}^{-1} (f)$. We discover in this construction the continuous triple functional calculus at the element $a$. \smallskip

Taking $f_n (t):= t^{\frac{1}{2n-1}}$ ($n\in \mathbb{N}$), the element $a^{[\frac{1}{2n-1}]} = \left(f_n\right)_t (a)$ is called the $(2n-1)$th root of $a$, and satisfies $(a^{[\frac{1}{2 n-1}]})^{[2n-1]} = a$ for every natural $n$. The sequence $(a^{[\frac{1}{2 n-1}]})_n$ converges in the weak$^*$
topology of $E^{**}$ to a (unique) tripotent in $E^{**}$, which is denoted by $r(a)$ and called the \emph{range tripotent} of $a$ in $E^{**}$. The tripotent $r(a)$ coincides with the smallest tripotent $e\in E^{**}$ satisfying that $a$ is a positive element in the JBW$^*$-algebra $(E^{**}_{2} (e), \circ_{e}, *_{e})$ (compare \cite[\S 3.3, Lemma
3.1]{EdwardsRutti1988facial} or \cite[Lemma 3.2. and the previous comments]{EdFerHosPe}).\smallskip

There is another way to locate (probably non-unital) JB$^*$-subalgebras inside JB$^*$-triples by employing the inner ideal generated by a single element. Suppose $a$ is a non-zero element in a JB$^*$-triple $E$. The symbol $E(a)$ will stand for the \emph{inner ideal of $E$ generated by the element $a$}. It is known (see \cite[Proposition 2.1]{Bunce2000structure}) that ${E}(a)$ coincides with the norm-closure of $\{a, {E}, a\} = Q(a) ({E})$ in ${E}$, and satisfies the following properties:
\begin{enumerate}[$(1)$]  
	\item ${E}_{a} \subset {E}(a).$
	\item ${E}(a)$ is a JB$^*$-subalgebra of ${E}^{**}_2(r(a))$.
	\item The weak$^*$-closure of ${E}(a)$ in ${E}^{**}$, $\overline{{E}(a)}^{w^*}$, identifies naturally with ${E}(a)^{**}$ and coincides with ${E}_{2}^{**}(r(a)).$
	\item Let $f:\Omega_a\to \mathbb{C}$ be a continuous function vanishing at zero, where $\Omega_a\subseteq (0,\|a\|]$ denotes the triple spectrum of $a$. Suppose that the function $\iota (t) =t$ ($t\in \Omega_a$) belongs to the subalgebra of $C_0(\Omega_a)$ generated by $f$. Then $E_{f_t(a)} = E_{a}$ and $E(a) = E(f_t(a))$. 
\end{enumerate}\smallskip

We are now ready to deal with general JB$^*$-triples and single-generated inner ideals.

\begin{proposition}\label{p inner ideal single generated is commutative} Let $E$ be a JB$^*$-triple for which there exists $\gamma>0$ satisfying \begin{equation}\label{eq Le Page triples propo}
		\big\|\{a,b,\{x,y,z\}\}\big\|\leq \gamma \ \! \big\|\{x,y,\{a,b,z\}\}\big\|, \hbox{ for all } a,b,x,y,z\in E.
	\end{equation} Then, for each $c\in E,$ the inner ideal $(E(c), \circ_{r(c)}, *_{r(c)})$ is an associative JB$^*$-algebra, equivalently, a commutative C$^*$-algebra.
\end{proposition}

\begin{proof} Fix $c\in E$, which can be clearly assumed to be non-zero. To simplify the notation we simply write $E(c)$ for $(E(c), \circ_{r(c)}, *_{r(c)})$. Define, via continuous triple functional calculus, a sequence given by $(c^{[\frac{1}{2 n +1}]})_n$. It is known that $(c^{[\frac{1}{2 n +1}]})_n$ converges in the weak$^*$-topology of $E^{**}$ to the range tripotent of $c$ in $E^{**}$ (see \cite[Lemma 3.1 and comments prior to it]{EdwardsRutti1988facial}).\smallskip
	
Clearly, $\left(c^{[\frac{1}{2 n +1}]}\right)_n\subseteq E_{c}\subseteq E(c),$ with $0\leq c^{[\frac{1}{2 n +1}]} \leq c^{[\frac{1}{2 m +1}]}\leq r(c)$ in the JBW$^*$-algebra $E(c)^{**} = E_2^{**} (r(c))$, for all $n\leq m$ in $\mathbb{N}$. \smallskip

We claim that \begin{equation}\label{ eq approximate unit for the inner ideal}\begin{aligned} \left(c^{[\frac{1}{2 n +1}]}\right)_n &
\hbox{ is a bounded approximate unit in the }\\
&\hbox{JB$^*$-algebra $(E(c), \circ_{r(c)}, *_{r(c)} )$,}
	\end{aligned}
\end{equation} that is, for each $d\in E(c)$, the sequence $\left(c^{[\frac{1}{2 n +1}]}\circ_{r(c)} d\right)_n\to d$ in norm. Namely, given a self-adjoint element $h\in E(c)_{sa}$, since the mapping $U_h: E(c)\to E(c)$ is a bounded linear and positive operator (cf. \cite[Proposition 3.3.6]{hanche-olsen84jordan}), we deduce that 
\[0\le U_h\left( c^{[\frac{1}{2 n +1}]} \right)\le U_h \left(c^{[\frac{1}{2 m +1}]}\right)\le U_h(r(c)) = h\circ_{r(c)} h =h^2 \in E(c),\] for all $m\ge n$. Note that the $n$th power of each element $z$ in the JB$^*$-algebra $E(c)$ will be denoted by $z^n$.\smallskip 

Consider the compact space $Q(E(c))$ of all quasi-states  on $E(c)$ (i.e. positive functionals on the JB$^*$-algebra $E(c)$ with norm $\leq 1$) equipped with the weak$^*$-topology of $E(c)^*$, that is, $Q(E(c)):= \{\varphi\in E(c)^*\ : \ \|\varphi\|\le 1\,\text{and}\,\varphi\ge 0\}$. Using similar ideas to those in Kadison's representation theorem, for each $z\in E(c)$, we define $\hat{z}: Q(E(c))\to\mathbb{C}$ by   $\hat{z}(\varphi)=\varphi(z)$. Note that $\hat{z}\in V$. Since $(c^{[\frac{1}{2 n +1}]})_n \nearrow r(c)$ in the weak$^*$-topology of $E^{**}$, it follows that $\left( \widehat{ U_h\left( c^{[\frac{1}{2 n +1}] } \right)}  \right)_n\nearrow \widehat{h^2}$ pointwise in $C(Q(E(c)))$. Dini's lemma assures that $\left( \widehat{ U_h\left( c^{[\frac{1}{2 n +1}] } \right)}  \right)_n \to \widehat{h^2}$ with respect to the supreme norm of $C(Q(E(c)))$. So, the sequence $\left( U_h\left( c^{[\frac{1}{2 n +1}] } \right)\right)_n$ converges in norm to $U_h (r(c))=h^2$.\smallskip

By applying \cite[Lemma 3.5.2$(ii)$]{hanche-olsen84jordan} with $a = r(c)- c^{[\frac{1}{2 n +1}] }$ and $b = h$ we have 
$$\begin{aligned}
\left\|\left(r(c)-c^{[\frac{1}{2 n +1}] }\right)\circ_{r(c)} h\right\|^2&\leq \left\|r(c)-c^{[\frac{1}{2 n +1}] } \right\| \ \left\|U_{h}\left(r(c)-c^{[\frac{1}{2 n +1}] }\right)\right\|\\ 
&\leq 2 \left\|U_{h}\left(r(c)-c^{[\frac{1}{2 n +1}] }\right)\right\| = 2 \left\| h^2 -  U_{h} \left(c^{[\frac{1}{2 n +1}] } \right)\right\| \to 0.
\end{aligned}$$

The above arguments show that for each element $h\in E(c)_{sa},$ the sequence $\left(c^{[\frac{1}{2 n +1}]}\circ_{r(c)} h\right)_n\to h$ in norm, and thus $\left(c^{[\frac{1}{2 n +1}]}\circ_{r(c)} d\right)_n\to d$ in norm for all $d\in E(c)$.\smallskip

We observe that the JB$^*$-algebra $(E(c), \circ_{r(c)}, *_{r(c)})$ admits two different triple products as JB$^*$-triple, namely, the original triple product of $E$ restricted to $E(c)$ and the triple product as JB$^*$-algebra given by $$\{x,y,z\} = (x\circ_{r(c)} y^{*_{r(c)}}) \circ_{r(c)} z + (z\circ_{r(c)} y^{*_{r(c)}})\circ_{r(c)} x -  (x\circ_{r(c)} z)\circ_{r(c)} y^{*_{r(c)}},$$ with $x,y,z\in E(c)$. Since linear bijections between JB$^*$-triples are isometries if, and only if, they are triple isomorphisms \cite[Proposition 5.5]{Kaup1983riemann}, both triple products coincide on $E(c)$. Therefore, for every $n\in \mathbb{N}$ and $x,y,z\in E(c)$ we have 
$$\begin{aligned}
\left\{x, c^{[\frac{1}{2 n +1}]} ,y\right\} &= \left(x\circ_{r(c)} \left(c^{[\frac{1}{2 n +1}]}\right)^{*_{r(c)}}\right) \circ_{r(c)} y \\
&+ \left(y\circ_{r(c)} \left(c^{[\frac{1}{2 n +1}]}\right)^{*_{r(c)}}\right) \circ_{r(c)} x -  (x\circ_{r(c)} y)\circ_{r(c)} \left( c^{[\frac{1}{2 n +1}]} \right)^{*_{r(c)}}.
\end{aligned}$$ It follows from the previous identity and \eqref{ eq approximate unit for the inner ideal} that the sequence $\left(\left\{x, c^{[\frac{1}{2 n +1}]} ,y\right\}\right)_n$ converges in norm to $x\circ_{r(c)} y$. Having in mind that the Jordan product jointly norm continuous, we can easily deduce that    
$$\left(\left\{\left\{x, c^{[\frac{1}{2 n +1}]} ,y\right\}, c^{[\frac{1}{2 n +1}]} ,z\right\}\right)_n \to (x\circ_{r(c)} y)\circ_{r(c)} z \hbox{ in norm.}$$
We can similarly prove that     
$$\left(\left\{\left\{z, c^{[\frac{1}{2 n +1}]} ,y\right\}, c^{[\frac{1}{2 n +1}]} ,x\right\}\right)_n \to (z\circ_{r(c)} y)\circ_{r(c)} x \hbox{ in norm.}$$ Since, by hypotheses \eqref{eq Le Page triples propo} $$\left\| \left\{\left\{x, c^{[\frac{1}{2 n +1}]} ,y\right\}, c^{[\frac{1}{2 n +1}]} ,z\right\} \right\| \leq \gamma \ \! \left\| \left\{\left\{z, c^{[\frac{1}{2 n +1}]} ,y\right\}, c^{[\frac{1}{2 n +1}]} ,x\right\} \right\|, $$ for all natural $n$, taking (norm-)limits we derive that $$ \left\|(x\circ_{r(c)} y)\circ_{r(c)} z \right\| \leq \gamma \ \! \left\| (z\circ_{r(c)} y)\circ_{r(c)} x \right\|,$$ for all $x,y,z\in E(c)$. Theorem~\ref{thm: associative JB*-algebra nilpotent} proves that $(E(c), \circ_{r(c)}, *_{r(c)})$ is an associative JB$^*$-algebra, as desired. 
\end{proof}

We know from \cite[Theorem 2$(i)\Leftrightarrow (xii)$]{CaPe24} that a JB$^*$-triple $E$ is commutative, if and only if, for each $c\in E$ the inner ideal $E(c)$ is an associative JB$^*$-algebra and $L(c,c) (E) \subseteq E(c)$. It is further known from \cite[Remark 3]{CaPe24} that the previous two conditions are required to conclude that $E$ is commutative. Actually, a complex Hilbert space $H$ regarded as a type 1 Cartan factor is a non-commutative JB$^*$-triple but the inner ideal generated by an arbitrary non-zero element in $H$ is isometrically isomorphic to $\mathbb{C}$. So, the conclusion in the previous Proposition~\ref{p inner ideal single generated is commutative} is not enough to deduce the commutativity of the JB$^*$-triple $E$. 

\section{Commutative JB$^*$-triples}\label{sec: commutative JB*-triples}

This section is devoted to prove Theorem~\ref{t characterization of commutative JB*-triples}. We begin by stating some necessary tools. Recall that a non-zero tripotent $e$ in a JB$^*$-triple $E$ is called minimal if $E_2(e) = \mathbb{C} e$. One of the consequence of the Sait\^o-Tomita-Lusin theorem for JB$^*$-triple established in \cite{BuFerMarPe2006} proves that finite sums of mutually orthogonal minimal tripotents in the bidual, $E^{**},$ of a JB$^*$-triple $E$ lie in $E$ up to an orthogonal perturbation. 

\begin{theorem}\label{t 3.3 Lusin}{\rm \cite[Theorem 3.3]{BuFerMarPe2006}} Let $E$ be a JB$^*$-triple, and let $u_1,\ldots , u_n$ be orthogonal minimal tripotents in $E^{**}$ with sum $u$. Then there exist norm-one pairwise orthogonal
	elements, $a_1,\ldots , a_n$ in $E$ such that $a_j = u_j +P_0 (u_j)  (a_j),$ for all $j=1,\ldots , n.$
\end{theorem}

Suppose $E$ is a weak$^*$-dense JB$^*$-subtriple of a JBW$^*$-triple $W$. The natural embedding of $E$ into $W$ (denoted by $\iota_{E}: E\hookrightarrow  W$) is an isometric triple embedding with weak$^*$-dense image. By \cite[Theorem 4 and Proposition 6]{BarDaHor1988}, there exist orthogonal weak$^*$-closed triple ideals $M,N$ of $E^{**},$ and a weak$^*$-continuous triple isomorphism $\Psi : W\to M$ such that $E^{**} = M \oplus N$, and the mapping $\Psi\circ \iota_{E}:E\to M$ coincides with the composition of the canonical embedding of $E$ inside $E^{**}$ (denoted by $\kappa_{E} : E\hookrightarrow E^{**}$) and the canonical projection $\pi_{M}$ of $E^{**}$ onto $M$. Take now a minimal tripotent $e$ in $W$. By applying Theorem~\ref{t 3.3 Lusin} to the minimal tripotent $\Psi (e)\in M,$ which is also minimal in $E^{**}$ with $\pi_{M} \Psi (e) = \Psi (e)$, we obtain a norm-one element $a\in E$ such that $\kappa_{E}(a) = \Psi(e) +P_0(\Psi(e)) (\kappa_{E}(a))$ in $E^{**}$. Therefore, $$\begin{aligned}
\Psi (\iota_{E} (a)) &= \pi_{M} (\kappa_{E}(a)) = \pi_{M} (\Psi(e)) +\pi_{M}(P_0(\Psi(e)) (\kappa_{E}(a))) \\
&= \Psi (e) +P_0(\Psi(e)) (\pi_{M}(\kappa_{E}(a))) = \Psi (e) +P_0(\Psi(e)) (\Psi(\iota_{E}(a)))\\
&= \Psi (e + P_0(e) (\iota_{E}(a))).
\end{aligned}$$ If we identify elements in $E$ with their images inside $W$, we get the following result.

\begin{theorem}\label{t 3.3 Lusin single element}{\rm \cite[Theorem 3.3]{BuFerMarPe2006}} Let $E$ be a weak$^*$-dense JB$^*$-subtriple of a JBW$^*$-triple $W,$ and let $e$ be a minimal tripotent in $W$. Then there exists a norm-one element $a$ in $E$ such that $a = e + P_0 (e) (a)$.
\end{theorem}

We are now in a position to prove the main result of this paper.

\begin{proof}[Proof of Theorem~\ref{t characterization of commutative JB*-triples}] Let $E$ be a JB$^*$-triple for which there exists a positive constant $\gamma$ satisfying the Le Page-type inequality in \eqref{eq Le Page triples}.  We have already seen in the previous section that, under these hypotheses, for each $c\in E$,  the inner ideal $E(c)$ generated by the element $c$ is an associative JB$^*$-algebra, equivalently, a commutative C$^*$-algebra (Proposition~\ref{p inner ideal single generated is commutative}). Arguing by contradiction we assume that $E$ is not commutative.\smallskip
	
It follows from \cite[Corollary 1$(a)\Leftrightarrow(b)$ or Theorem 2]{CaPe24} and what we have just commented in the previous paragraph that, under our assumptions, if $E$ is not commutative, all Cartan factors in $E^{**}$ are rank-one Hilbert spaces and at least one of them has dimension $\geq 2$. It follows that the atomic part of $E^{**}$ is of the form $\mathcal{A} = \bigoplus_{i\in \Gamma}^{\infty} H_i,$ where each $H_i$ is a complex Hilbert space (cf. \cite[Proposition 2]{FriRuss86} and \cite[Corollary 1.8]{horn87classificationtypeI}).\smallskip

A cornerstone in the theory of JB$^*$-triples assures that for each JB$^*$-triple $E,$ there exists an isometric triple embedding with weak$^*$-dense image $\Psi_{E}: E\to \mathcal{A},$ where $\mathcal{A}$ denotes the atomic part of $E^{**}$, which is known to be representable as a direct sum of Cartan factors (see \cite[Propositions 1 and 2 and Theorem 1]{FriRuss86} and \cite[Theorem 1]{FriRuss85Crelles}).\smallskip

Combining the conclusions in the above paragraphs, we deduce that if $E$ is a non-commutative JB$^*$-triple satisfying \eqref{eq Le Page triples}, there exists a family of complex Hilbert spaces $\{ H_i,  : i\in \Gamma \}$ with $\Gamma_2 := \{i\in \Gamma: \hbox{dim} (H_i)\geq 2\}\neq \emptyset,$ and an isometric triple embedding with weak$^*$-dense image $\Psi_{E} : E\hookrightarrow  \bigoplus_{i\in \Gamma}^{\infty} H_i$. The inner product of each Hilbert space $H_i$ will be denoted by $\langle.|.\rangle_i$ or simply by $\langle.|.\rangle$ if the meaning is clear. To simplify the notation we write $\mathcal{A} = \bigoplus_{i\in \Gamma}^{\infty} H_i$, and note that $\mathcal{A}$ is a Hilbert $\ell_{\infty}(\Gamma)$-bimodule with respect to the pointwise operation and inner product given by $\langle (a_i)_{i\in \Gamma} | (b_i)_{i\in \Gamma}\rangle := (\langle a_i| b_i\rangle_i )_{i\in \Gamma}\in \ell_{\infty}(\Gamma)$. The triple product on $\mathcal{A}$ is given by \begin{equation}\label{eq triple product of A} \{(x_i)_{i\in \Gamma}, (y_i)_{i\in \Gamma}, (z_i)_{i\in \Gamma} \}_{_\mathcal{A}}\! =\! \Big(\{x_i, y_i, z_i \}\Big)_{i\in \Gamma}\!\!\!  =\! \left(\frac12 \langle x_i| y_i \rangle_{i} z_i + \frac12 \langle z_i| y_i \rangle_{i} x_i   \right)_{i\in \Gamma}\!\!\!.
\end{equation} We can clearly identify $E$ with its image under $\Psi_{E}$.\smallskip

Let us make an observation. Given a norm-one element $a= (a_i)_{i\in \Gamma}\in\mathcal{A}$, its range tripotent in $\mathcal{A}$ is precisely $r= (r_i)_{i\in \Gamma}$, where $r_i =0 $ if $a_i =0$ and $r_i = \frac{a_i}{\|a_i\|}$ otherwise.\smallskip

Pick $i_0\in \Gamma_2$ and two norm-one elements $h,k\in H_{i_0}$ with $\langle h | k \rangle_{i_0}=0$. We decompose $\mathcal{A}$ in the form $\mathcal{A} = H_{i_0} \bigoplus^{\infty} \left( \bigoplus_{i_0\neq i\in \Gamma}^{\infty} H_i \right)$. The elements $\mathbf{h} = (h,0)$ and $\mathbf{k} = (k,0)$ are minimal tripotents in $\mathcal{A}$. By Theorem~\ref{t 3.3 Lusin single element} there exist norm-one elements $a,b\in E$ satisfying $a= \mathbf{h}+ P_0 (\mathbf{h}) (a)$ and $b= \mathbf{k}+ P_0 (\mathbf{k}) (b)$. Let $r =(h, (r_i)_{i_0\neq i\in \Gamma}) = \mathbf{h} + r_0 = \mathbf{h} + P_0 (\mathbf{h}) (r)   \in\mathcal{A}$ denote the range tripotent of $a$ in $\mathcal{A}$.\smallskip

We claim that $b\notin E(a)$. Otherwise, since $(E(a), \circ_r, *_{r})$ is a JB$^*$-algebra we would have $$\begin{aligned}
	\mathbf{k} + P_0(\mathbf{k}) (b)  &= \left(k, (b_i)_{i\neq i_0} \right) =  b = (b^{*_r})^{*_r} = \{r,\{r,b,r\},r\}  \\
	&= \Big( \{h,\{h,k,h\},h\}, (\{r_i,\{r_i,b_i,r_i\},r_i\})_{i\neq i_0} \Big)\\
	&= \Big( 0, (\{r_i,\{r_i, b_i, r_i\},r_i\})_{i\neq i_0} \Big), 
\end{aligned} $$ which is impossible.\smallskip

Making use of the continuous triple functional calculus, we can easily check that for each $0<\beta <1$ and $a_0 = P_0(\mathbf{h})(a)$, we have \begin{equation}\label{eq a to the beta} a^{[\beta]} = (\|a_i \|^{\beta} \ r_i)_{ i} = (h, (\|a_i \|^{\beta} \ r_i)_{i\neq i_0})= \mathbf{h} + a_0^{[\beta]}.
\end{equation} Since $b\notin E(a),$ $a^{[\beta]}\in E_{a}\subseteq E(a)= E(a^{[\beta]})$, and thus $\{a^{[\beta]}, b, a^{[\beta]}\}\in E(a)$, we deduce that $E\ni c_{\beta} = b - \{a^{[\beta]}, b, a^{[\beta]}\}^{*_{r}}\notin E(a)$. Obviously, $\|c_{\beta}\|\leq 2$.\smallskip

We shall next prove that for $\beta$ close to $0$, the element $c_{\beta}$ is almost orthogonal to $a$ in the Hilbert $\ell_{\infty}(\Gamma)$-bimodule $\mathcal{A}$. More concretely \begin{equation}\label{eq almost orthogonality of a and cbeta}  \left\| \langle a | c_{\beta} \rangle \right\| \leq \frac{2 \beta}{(1 + 2 \beta)^{\frac{1}{2 \beta} + 1}}, \hbox{ for all } 0<\beta<1. 
\end{equation}  Namely, having in mind \eqref{eq a to the beta} and the triple product of $\mathcal{A}$ in \eqref{eq triple product of A} we get $$E(a) \ni \{a^{[\beta]}, b, a^{[\beta]}\} = \left( \|a_i\|^{2 \beta} \langle r_i | b_i  \rangle r_i  \right)_{i},$$ and $$E(a) \ni \{a^{[\beta]}, b, a^{[\beta]}\}^{*_{r}}  = \{r, \{a^{[\beta]}, b, a^{[\beta]}\} ,r\}= \left( \|a_i\|^{2 \beta} \langle b_i | r_i  \rangle r_i  \right)_{i},$$ where if $r_i\neq 0$ we have $\|a_i\|^{2 \beta-2} \langle b_i | a_i  \rangle a_i.$ Therefore, $$ \langle c_{\beta} | a\rangle  = \Big( \langle b_i | a_i \rangle - \|a_i\|^{2 \beta} \langle b_i | r_i  \rangle \langle r_i | a_i  \rangle \Big)_{i}.$$ If $a_i \neq 0$, the $i$th component of $\langle c_{\beta} | a\rangle$ is $$\begin{aligned}
\langle b_i | a_i \rangle - \|a_i\|^{2 \beta} \langle b_i | r_i  \rangle \langle r_i | a_i  \rangle &= \langle b_i | a_i \rangle - \|a_i\|^{2 \beta} \frac{1}{\|a_i\|} \langle b_i | {a_i}  \rangle \frac{1}{\|a_i\|} \langle {a_i} | a_i  \rangle \\
&= \langle b_i | a_i \rangle - \|a_i\|^{2 \beta} \langle b_i | {a_i}  \rangle,
\end{aligned}$$ and hence $$ \|  \langle c_{\beta} | a\rangle\| \leq \sup_{i} \|b_i\| \|a_i\| \left( 1- \|a_i\|^{2 \beta} \right) \leq \sup_{i} \|a_i\| \left( 1- \|a_i\|^{2 \beta} \right).$$ The desired inequality in \eqref{eq almost orthogonality of a and cbeta} follows from the fact that the maximum value of the function $f(t) = t \left(1 - t^{2\beta} \right)$ on the interval $[0,1]$ is $\frac{2 \beta}{(1 + 2 \beta)^{\frac{1}{2 \beta} + 1}}$. Note that $\displaystyle \lim_{\beta\to 0^+} \frac{2 \beta}{(1 + 2 \beta)^{\frac{1}{2 \beta} + 1}} =0$.\smallskip

In the decomposition $\mathcal{A} = H_{i_0} \bigoplus^{\infty} \left( \bigoplus_{i_0\neq i\in \Gamma}^{\infty} H_i \right)$, we can write $$\begin{aligned}
c_{\beta} = b- \{a^{[\beta]}, b, a^{[\beta]}\} &= \left(
k -\{h,k,h\}, ((c_{\beta})_i)_{i\neq i_0} \right) \\
&= \left(
k , ((c_\beta)_i)_{i\neq i_0} \right)=\mathbf{k} + P_0(\mathbf{k}) (c_{\beta}).
\end{aligned}$$ 

Observe that $$\begin{aligned}
\langle c_{\beta}| a \rangle &= \langle \mathbf{k} + P_0(\mathbf{h}) (c_{\beta}) | \mathbf{h} + P_0(\mathbf{h}) (a) \rangle \\
&= \langle \mathbf{k}  | \mathbf{h}  \rangle + \langle  P_0(\mathbf{h}) (c_{\beta}) | P_0(\mathbf{h}) (a) \rangle= \langle  P_0(\mathbf{h}) (c_{\beta}) | P_0(\mathbf{h}) (a) \rangle = \langle  P_0(\mathbf{h}) (c_{\beta}) | a_0 \rangle.
\end{aligned}.$$

We finally consider the triple products $\{c_{\beta},a,\{a,c_{\beta},a\}\}$ and $\{a, c_{\beta},\{c_{\beta},a,a\}\}$ and compute their norms. In the first case we have
$$\begin{aligned}
\{c_{\beta}&,a,\{a,c_{\beta},a\}\} = \{\mathbf{k} + P_0(\mathbf{k}) (c_{\beta}),\mathbf{h}+a_0, \{\mathbf{h}+a_0,\mathbf{k} + P_0(\mathbf{k}) (c_{\beta}),\mathbf{h}+a_0\}\} \\
&= \{ P_0(\mathbf{k}) (c_{\beta}),  a_0 , \{  a_0,  P_0(\mathbf{k}) (c_{\beta}), a_0\}\} = \{ P_0(\mathbf{k}) (c_{\beta}),  a_0, \langle  a_0 |  P_0(\mathbf{k}) (c_{\beta})\rangle a_0\}\\
&= \frac12 \langle P_0(\mathbf{k}) (c_{\beta})|  a_0 \rangle \langle  a_0 |  P_0(\mathbf{k}) (c_{\beta})\rangle a_0 
+ \frac12 \langle  a_0 |  P_0(\mathbf{k}) (c_{\beta})\rangle \langle a_0| a_0\rangle P_0(\mathbf{k}) (c_{\beta}) \\
& = \frac12 \langle c_{\beta} |  a \rangle \langle  a |  c_{\beta}\rangle a_0 + \frac12 \langle  a |  c_{\beta} \rangle \langle a_0| a_0\rangle P_0(\mathbf{k}) (c_{\beta}), 
\end{aligned} $$  with $$\| \{c_{\beta},a,\{a,c_{\beta},a\}\} \|\leq \frac12 \| \langle c_{\beta} |  a \rangle \langle  a |  c_{\beta}\rangle\| + \| \langle  a |  c_{\beta} \rangle \| \leq  2 \| \langle  a |  c_{\beta} \rangle \| \leq \frac{4 \beta}{(1 + 2 \beta)^{\frac{1}{2 \beta} + 1}}.$$

In the second case, by orthogonality, we obtain 
$$ \begin{aligned}
\{a, c_{\beta},\{c_{\beta},a,a\}\}  &=  \{\mathbf{h}+a_0, \mathbf{k} + P_0(\mathbf{h}) (c_{\beta}),\{\mathbf{k} + P_0(\mathbf{h}) (c_{\beta}),\mathbf{h}+a_0,\mathbf{h}+a_0\}\} \\
&= \{\mathbf{h}+a_0, \mathbf{k} + P_0(\mathbf{h}) (c_{\beta}),\{\mathbf{k},\mathbf{h},\mathbf{h}\} +  \{ P_0(\mathbf{h}) (c_{\beta}),a_0,a_0\}\} \\
&= \left\{\mathbf{h}+a_0, \mathbf{k} + P_0(\mathbf{h}) (c_{\beta}),\frac12 \mathbf{k}+  \{ P_0(\mathbf{h}) (c_{\beta}),a_0,a_0\}\right\} \\
&= \frac12 \left\{\mathbf{h}, \mathbf{k}, \mathbf{k} \right\} + \left\{ a_0, P_0(\mathbf{h}) (c_{\beta}),  \{ P_0(\mathbf{h}) (c_{\beta}),a_0,a_0\}\right\} \\
&= \frac14 \mathbf{h} + \left\{ a_0, P_0(\mathbf{h}) (c_{\beta}),  \{ P_0(\mathbf{h}) (c_{\beta}),a_0,a_0\}\right\},
\end{aligned} $$ and $$ \| \{a, c_{\beta},\{c_{\beta},a,a\}\}\| = \max\left\{\frac14 \|\mathbf{h} \|, \left\| \left\{ a_0, P_0(\mathbf{h}) (c_{\beta}),  \{ P_0(\mathbf{h}) (c_{\beta}),a_0,a_0\}\right\} \right\| \right\} \geq \frac14.$$ Taking $\beta$ close enough to $0$ to assure that $\gamma\  \frac{4 \beta}{(1 + 2 \beta)^{\frac{1}{2 \beta} + 1}} <\frac{1}{26}$, and considering the hypothesis \eqref{eq Le Page triples}, we obtain $$\frac{1}{26} > \gamma \ \! \| \{c_{\beta},a,\{a,c_{\beta},a\}\} \| \geq   \| \{a, c_{\beta},\{c_{\beta},a,a\}\}\|  \geq \frac14,$$ which is impossible. We have arrived to the desired contradiction, and hence $E$ must be commutative. 
\end{proof}

\medskip
\medskip

\textbf{Acknowledgements} L. Li was supported by National Natural Science Foundation of China (grant No. 12571143). A.M. Peralta has been supported by MICIU/AEI/10.13039/501100011033 and ERDF/EU through the grant PID2021-122126NB-C31, by ``Maria de Maeztu'' Excellence Unit IMAG, reference CEX2020-001105-M, and by the Bureau of Foreign Experts Affairs, MOHRSS, PRC China (grant S20250924). \smallskip

This work was partly carried out during a visit of A.M. Peralta to Nankai University and the Chern Institute of Mathematics (April 2025), whose hospitality is gratefully acknowledged.\medskip

We thank the referee for the useful comments and suggestions.

\subsection*{Statements and Declarations} 

All authors declare that they have no conflicts of interest to disclose.

\subsection*{Data availability}

There is no data associated with this submission.



\end{document}